\documentclass[final,twoside,11pt]{entics} 
\usepackage{enticsmacro}
\usepackage{graphicx}
\usepackage[all]{xy}
\newtheorem{defi}[thm]{Definition}
\newtheorem{ex}[thm]{Example}

\newcommand{\cl}{\operatorname{cl}}             
\newcommand{\Int}{\operatorname{int}}           

\newcommand{\fin}{\operatorname{fin}}

\newcommand{\sm}{\setminus}
\makeatletter
\newcommand{\rmnum}[1]{\romannumeral #1}
\newcommand{\Rmnum}[1]{\expandafter\@slowromancap\romannumeral #1@}
\makeatother

\newcommand\twoheaduparrow{\mathord{\rotatebox[origin=c]{90}{$\twoheadrightarrow$}}}
\newcommand\twoheaddownarrow{\mathord{\rotatebox[origin=c]{90}{$\twoheadleftarrow$}}}
\newcommand{\dda}{\twoheaddownarrow}
\newcommand{\dua}{\twoheaduparrow}
\newcommand{\ua}{\uparrow\nobreak\!}
\newcommand{\da}{\downarrow\nobreak\!}

\def\conf{ISDT 2022} 	
\volume{2}			
\def\volu{2}			
\def\papno{3}			
\def\mydoi{{\normalshape  \href{https://doi.org/10.46298/entics.10383}{10.46298/entics.10383}}}
\def\copynames{Z. He, Z. Yang, D. Zhao, X. Jia} 
\def\runtitle{Quasiexact posets and the moderate meet continuity}

\def\lastname{He, Yang and Zhao}
\begin{document}
\begin{frontmatter}
  \title{Quasiexact Posets and the Moderate Meet-continuity} 
  \author{Zhaorong He\thanksref{a}\thanksref{ALL}\thanksref{myemail}}	
   \author{Zhongqiang Yang\thanksref{b}\thanksref{coemail}}
   \author{Dongsheng Zhao\thanksref{c}\thanksref{coemail-1}}		
   \address[a]{Department of Mathematics\\ Shantou University\\				
    Shantou, China 515063}  							
  \thanks[ALL]{The first two authors were supported by National Natural Science Foundation of China (No. 11971287). This work was completed while the first author visiting Nanyang Technological University under the Chinese Government Scholarship (No. 201908440290). }   
   \thanks[myemail]{Email: \href{mailto:17zrhe@stu.edu.cn} {\texttt{\normalshape
        17zrhe@stu.edu.cn}}} 
  \address[b]{School of Mathematics and Statistics\\Minnan Normal University\\
    Zhangzhou, China 363000} 
  \thanks[coemail]{Email:  \href{mailto:zqyang@stu.edu.cn} {\texttt{\normalshape
        zqyang@stu.edu.cn}}}
  \address[c]{Department of Mathematics and Mathematics Education, National Institute of Education\\Nanyang Technological University\\
    1 Nanyang Walk, Singapore 637616}
   \thanks[coemail-1]{Email:  \href{mailto:dongsheng.zhao@nie.edu.sg} {\texttt{\normalshape
        dongsheng.zhao@nie.edu.sg}}}    
    
\begin{abstract} 
The study of weak domains and quasicontinuous domains leads to the consideration of two types  generalizations of domains. In the current paper, we define the weak way-below relation between two nonempty subsets of a poset and  quasiexact posets. We prove some connections among quasiexact posets, quasicontinuous domains and weak domains. Furthermore, we introduce the weak way-below finitely determined topology and study its links to Scott topology and the weak way-below topology first considered by Mushburn. It is also proved that a dcpo is a domain if it is quasiexact and moderately meet continuous with the weak way-below relation weakly increasing.
\end{abstract}

\end{frontmatter}

\section{Introduction}\label{intro}
Scott \cite{Scott-1970} proposed a model for information systems using the Scott topology  and a binary  relation $\prec$ in connection with the information models. For continuous lattices, the relation $\prec$ coincides with the way-below relation. The class of continuous complete lattices was introduced by Scott \cite{Scott-1974}. However, for general complete lattices, the aforementioned two relations may  be distinct.

One of the notable features of continuous lattices is that they admit a unique compact Hausdorff topology for which the meet operation is continuous. This topology, referred to as the CL-topology \cite{Gierz-Lawson-1981}, turns out to be `order intrinsic' - it  can be defined merely  using  the lattice structure. Gierz and Lawson \cite{Gierz-Lawson-1981} characterized those complete lattices for which the CL-topology is Hausdorff and called them generalized continuous lattices. Gierz et al. \cite{Gierz-Lawson-Stralka-1983b} introduced the quasicontinuous posets and showed that a complete lattice is  generalized continuous if and only if it is quasicontinuous. The key result for establishing the major properties of quasicontinuous dcpos is the Rudin's Lemma \cite{Rudin-1981}.

Coecke and Martin \cite{Coecke-Martin-2002} introduced two orders: the Bayesian order on classical states and the spectral order on quantum states. They revealed that the corresponding  sets are dcpos with an intrinsic notion of approximation. The operational significance of the orders involved conclusively establishes that physical information has a natural domain-like theoretic structure. Mushburn \cite{Mushburn-2007} called the approximation in  \cite{Coecke-Martin-2002} the weak way-below relation, and defined two topologies on posets: the way-below topology and the weak way-below topology. These topologies coincide with  the Scott topology for continuous posets,  but are very different for non-continuous posets. Mushburn also showed that while domain representable spaces must be Baire,  this is not the case with respect to the new topologies. Thus, Mushburn defined the  weak domains  and weak domain representable spaces and constructed an example to show that weak domain representable spaces need not be Baire \cite{Mushburn-2010}.

The class of meet continuous lattices was first introduced by Birkhoff \cite{Birkhoff-1967b}. Later, much investigations on meet continuity for lattices and semilattices sprang up. One can refer to Isbell \cite{Isbell-1975b}, Hofmann and Stralka \cite{Hofmann-Stralka-1976} and \cite{Gierz-book-2003}.  Kou et al. \cite{Kou-Liu-Luo-2001} extended the notion of meet continuity to general dcpos and proved that a dcpo is continuous iff it is  meet continuous and quasicontinuous. See \cite{Mao-Xu-2006} and \cite{Mao-Xu-2009} for the investigation of the more general meet-continuous posets and quasicontinuous posets.

As a generalization of the way-below relation between two subsets of posets and the weak way-below relation between elements, we define the weak way-below relation between two subsets of a poset and use this relation to define the quasiexact posets. We show some connections among quasiexact posets, quasicontinuous domains and weak domains. Furthermore, we introduce the weak way-below finitely determined topology (briefly, wf topology) and study its properties as well as its links to Scott topology and weak way-below topology in some special posets. In addition, we prove that a poset is a weak domain if it is moderately meet continuous and quasiexact dcpo with the relation $\ll_{w}$ weakly increasing, a result similar to the characterization of domains in terms of meet continuity and quasicontinuity. Finally, employing a result by Shen et al. \cite{Shen-Wu-Zhao-2019}, it is deduced that a dcpo is a domain if and only if it is quasiexact and moderately meet continuous with the weak way-below relation weakly increasing.

\section{Preliminaries}

In this section we recall some  notations, definitions and results to be used later.

For any subset $A$ of a poset $P$, we write $\ua A=\{x\in P: y\leq x~\mbox{for some}~y\in A\}$. A subset $A\subseteq P$ is called an \emph{upper set} if $\ua A=A$.

A nonempty subset $D$ of a poset $(P,\leq)$ is \emph{directed} if every two
elements in $D$ have an upper bound in $D$. A poset $P$ is called a \emph{directed complete} poset (\emph{dcpo}, for short) if every directed subset $D$ of $P$ has a supremum in $P$, denoted by $\bigvee D$.

For any poset $P$, the \emph{way-below} relation $\ll $ on  $P$ is defined as follows: for every directed subset $D$ with $\bigvee D$ existing, $y\leq \bigvee D$ implies  $x\leq d$ for some $d\in D$. A poset $P$ is \emph{continuous} if for every $x\in P$, the set $\dda x=\{y\in P: y\ll x\}$ is directed and $x=\bigvee\dda x$. For any $x\in P$, one writs  $\dua x=\{y\in P: x\ll y\}$.

 A subset $U$ of a poset $P$ is \emph{Scott open} if $U$ is an upper set and for any directed subset $D$ of $P$ for which $\bigvee D$ exists, $\bigvee D\in U$ implies $D\cap U\neq\varnothing$.
All Scott open subsets of $P$ form a topology,
called the \emph{Scott topology} on $P$ and
denoted by $\sigma(P)$. The space $(P,\sigma(P))$ is called the
\emph{Scott space} of $P$, and is denoted by $\Sigma P$.

A poset $P$ is said to be \emph{meet continuous}\footnote{In general, the meet continuity means for dcpos. In the current paper, we slightly misuse this notion.} if for any $x\in P$ and any directed subset $D$, $x\leq \bigvee D$ implies $x\in \cl_{\sigma}(\da D\cap\da x)$.

For any topological space $(X, \tau)$ and a subset $A\subseteq X$, the closure and the interior of $A$ are denoted by $\cl_{\tau}(A)$ and $\Int_{\tau}(A)$, respectively.

Let $(X, \tau)$ be a topological space. A nonempty subset $A$ of $X$  is called \emph{irreducible} if for any closed sets  $B$ and $C$, $A\subseteq B\cup C$  implies $A\subseteq B$ or $A\subseteq C$. The space $X$ is  \emph{sober} if for every irreducible closed set $A$, there exists a unique $x\in X$ such that $\cl_{\tau}\{x\}= A$.

The weak way-below relation on a poset $P$ is defined as follows \cite{Mushburn-2007}: for any $x, y\in P$,  $x$ is called \emph{weakly way below} $y$, denoted by $x\ll_{w} y$ if for any directed subset $D$ of $P$, $\bigvee D=y$ implies $D\cap \ua x\neq \varnothing$.

Note that for continuous posets, the relations $\ll_{w}$ and $\ll$ are the same. For any $x\in P$, we write $\dda_{w} x=\{y\in P: y\ll_{w} x\}$. Mushburn \cite{Mushburn-2007} pointed out the following properties.

\begin{lem} \emph{\cite[Theorem 3.1]{Mushburn-2007}} \label{Lemma-Mushburn}
For any elements  $x, y, z$ in a poset $P$,
\begin{enumerate}
\renewcommand{\labelenumi}{(\theenumi)}
\item If $x\ll y$, then $x\ll_{w} y$;
\item If $x\ll_{w} y$, then $x\leq y$;
\item If $x\leq y\ll_{w}z$, then $x\ll_{w}z$;
\item If $P$ has the bottom element $\bot$, then $\bot\ll_{w}x$.
\end{enumerate}
\end{lem}

It is well-known that the way-below relation $\ll$ is transitive: $x\ll y\leq z$ implies $x\ll z$. Whereas, this property may not be true for $\ll_{w}$. In fact, Coecke and Martin \cite{Coecke-Martin-2002} showed that the transitivity is true for $\ll_{w}$ if and only if $\ll_{w}=\ll$.

 A poset $P$ is called \emph{exact} if for any $x\in P$, the set $\dda_{w}x$ is directed and $\bigvee \dda_{w}x=x$ \cite{Mushburn-2007}. The relation $\leq_{w}$ on a poset $P$ is said to be \emph{weakly increasing} if for any $x, y, z, u\in P$, $x\ll_{w}y\leq z\ll_{w}u$ implies $x\ll_{w}z$. A poset $P$ is called a \emph{weak domain} if it is an exact dcpo with the relation $\ll_{w}$ weakly increasing.

Shen et al. \cite{Shen-Wu-Zhao-2019} proved the following characterization of exact dcpos.

\begin{prop} \label{equivalent condition of exact} \emph{\cite{Shen-Wu-Zhao-2019}}
A dcpo $P$ is exact if and only if for any $x\in P$, there exists a directed subset $D\subseteq\dda_{w}x$ such that $\bigvee D=x$.
\end{prop}

For a poset $P$, the collection of all nonempty subsets of $P$ is denoted by $\mathcal{P}^{*}(P)$. A preorder $\leq$ on $\mathcal{P}^{*}(P)$ is defined by $G\leq H$ if $\ua H\subseteq \ua G$. This preorder is sometimes called the Smyth order, see \cite{Smyth-1978}.

A nonempty family $\mathcal{F}$ of subsets of a poset $P$ is said to be \emph{directed} if given $F_{1}$, $F_{2}\in\mathcal{F}$, there exists $F_{3}\in\mathcal{F}$ such that $F_{1},~F_{2}\leq F_{3}$, that is  $F_{3}\subseteq \ua F_{1}\cap\ua F_{2}$, or equivalently, $\ua F_{3}\subseteq\ua F_{1}\cap \ua F_{2}$. For any $G, H\subseteq P$, $G$ is \emph{way below} $H$,  written as $G\ll H$,  if for every directed subset $D\subseteq P$, $\sup D\in \ua H$ implies $d\in \ua G$ for some $d\in D$. Sometimes, one writes $G\ll x$ instead of $G\ll \{x\}$ and $y\ll H$ instead of $\{y\}\ll H$. For more details one can refer to \cite{Gierz-book-2003}.

Recall that a dcpo $P$ is called a \emph{quasicontinuous domain} if for each $x\in P$ the family $$\fin(x)=\{F\subseteq P: F~~\mbox{is finite,~~}F\ll x\}$$ is directed and whenever $x\nleq y$, then there exists $F\in\fin(x)$ with $y\notin\ua F.$ This statement is equivalent to for each $x\in P$ the family $$\fin(x)=\{F\subseteq P: F~~\mbox{is finite,~~}F\ll x\}$$ is directed and $\bigcap\{\ua F: F\in\fin(x)\}=\ua x.$

For more about posets and related notions and results, we refer to \cite{Gierz-book-2003} and \cite{Goubault-book-2013}.

\section{Quasiexact dcpos}

We first introduce the following definition, as the generalization of both the way-below relation between two nonempty subsets and the weak way-below relation between two points.

\begin{defi}
For any poset $P$ and $G, H\in \mathcal{P}^{*}(P)$, we say that $G$ is \emph{weakly way below} $H$ and write $G\ll_{w} H$, if for every directed subset $D\subseteq P$, $\bigvee D\in H$ implies  $d\in \ua G$ for some $d\in D$.
\end{defi}

The following properties can be verified straightforwardly.

\begin{lem}\label{Lemma-3.2}
For any poset $P$, let $G, G', H, H'\in \mathcal{P}^{*}(P)$ and $x\in P$. Then
\begin{enumerate}
\renewcommand{\labelenumi}{(\theenumi)}
  \item $G\ll_{w} H$ if and only if $G\ll_{w}x$ for all $x\in H$;
  \item $G\ll_{w}H$ if and only if $\ua G\ll_{w}H$;
  \item $G\ll_{w}H$ and $G\subseteq G'$ imply $G'\ll_{w}H$;
  \item $G\ll_{w}H$ and $H'\subseteq H$ imply $G\ll_{w}H'$.
\end{enumerate}
\end{lem}

Thereafter, we write $G\ll_{w} x$ instead of $G\ll_{w} \{x\}$ and $y\ll_{w} H$ instead of $\{y\}\ll_{w} H$. By Lemma \ref{Lemma-3.2}, $y\ll_{w}x$ is unambiguously defined and this fact is similar with the one for the way-below relation.

For any poset $P$ and $x\in P$, write $$\fin_{w}(x)=\{F\subseteq P: F ~\mbox{is finite}, F\ll_{w}x\}.$$

\begin{defi}
A poset $P$ is said to be \emph{quasiexact} if for each $x\in P$, the family $\fin_{w}(x)$ is directed and $\bigcap\{\ua F: F\in \fin_{w}(x)\}=\ua x$.
\end{defi}

\begin{rem}\label{Remark-3.1}
For any poset $P$, it is easy to verify that $F\ll_{w} x$ implies  $x\in \ua F$ for any  $F\subseteq P$ and $x\in P$. Thus, $\bigcap\{\ua F: F\in \mathcal{F}\}\supseteq\ua x$  holds for all $x\in P$ and $\mathcal{F}\subseteq\fin_{w}(x)$. In particular, $\bigcap\{\ua F: F\in \fin_{w}(x)\}\supseteq\ua x$ for any $x\in P$.
\end{rem}

Sometimes it is difficult to characterize the weak way-below relation on a poset. To show a poset $P$ is quasiexact, it is sufficient to know for any $x\in P$, there are `enough' elements that are weakly way below $x$. In order to elaborate this fact, we first show the following lemmas.

The proofs of the following Lemmas \ref{equivalence of weak way-below} and \ref{F intersection down x is directed} are straightforward.

\begin{lem}\label{equivalence of weak way-below}
For any poset $P$ and $F\in \mathcal{P}^{*}(P), x\in P$,   $F\ll_{w}x~\mbox{if and only if}~\da x\cap F\ll_{w}x.$
\end{lem}

\begin{lem} \label{F intersection down x is directed}
For any poset $P$, let $\mathcal{F}\subseteq\mathcal{P}^{*}(P)$ and $x\in P$. If $\mathcal{F}$ is directed, then $\{F\cap \da x: F\in\mathcal{F}\}$ is directed provided that each $F\cap \da x$ is nonempty.
\end{lem}

The following lemma given by Rudin \cite{Rudin-1981} is crucial in proving a number of major properties of  quasicontinuous domains, one can also refer to Lemma \uppercase\expandafter{\romannumeral 3}-3.3 in \cite{Gierz-book-2003}.

\begin{lem}\emph{(\textbf{Rudin's Lemma})}\label{Rudin's Lemma}
For any poset $P$, let $\mathcal{F}$ be a directed family of nonempty finite subsets of $P$. There exists a directed set $D\subseteq\bigcup_{F\in \mathcal{F}} F$ such that $D\cap F\neq\varnothing$ for every $F\in \mathcal{F}$.
\end{lem}

\begin{lem} \label{Lemma according to Rudin's Lemma}
For any dcpo $P$, let $G\in \mathcal{P}^{*}(P)$ and $x\in P$ with $G\ll_{w}x$. If $\mathcal{F}$ is a directed family  of nonempty finite subsets $F\subseteq \da x$ with $\bigcap_{F\in\mathcal{F}}\ua F\subseteq\ua x$, then there exists $F_{0}\in \mathcal{F}$ such that $F_{0}\subseteq\ua G$.
\end{lem}

\begin{proof}
Assume, on the contrary that $F\backslash\ua G\neq \varnothing$ for all $F\in \mathcal{F}$. For any $F_{1}, F_{2}\in\mathcal{F}$, choose $F_{3}\in \mathcal{F}$ such that $F_{3}\subseteq \ua F_{1}\cap \ua F_{2}$. Then $F_{3}\backslash \ua G\subseteq (\ua F_{1}\cap \ua F_{2})\backslash\ua G=(\ua F_{1}\backslash\ua G)\cap(\ua F_{2}\backslash\ua G)$. It is easy to verify that $\ua F_{1}\backslash\ua G\subseteq \ua (F_{1}\backslash \ua G)$ and $\ua F_{2}\backslash\ua G\subseteq \ua (F_{2}\backslash \ua G)$. Hence
$F_{3}\backslash \ua G\subseteq \ua(F_{1}\backslash\ua G)\cap\ua (F_{2}\backslash\ua G)$.
Thus, $\{F\backslash\ua G: F\in \mathcal{F}\}$ is a directed family of nonempty finite subsets. By Rudin's Lemma, there exists a directed set $D\subseteq\bigcup_{F\in\mathcal{F}}(F\backslash\ua G)$ such that $D\cap(F\backslash\ua G)\neq\varnothing$ for any $F\in \mathcal{F}$. For every $F\in \mathcal{F}$, choose $d_{F}\in D$ such that $d_{F}\in F\setminus\ua G$. Then $\bigvee D\in\bigcap_{F\in \mathcal{F}} \ua d_{F}\subseteq \bigcap_{F\in \mathcal{F}}\ua (F\setminus\ua G)\subseteq\bigcap_{F\in \mathcal{F}}\ua F\subseteq\ua x$. Thus, $\bigvee D\geq x$. Note that $D\subseteq\bigcup_{F\in\mathcal{ F}}F\subseteq \da x$, so $\bigvee D\leq x$. It follows that $\bigvee D=x$. By $G\ll_{w}x$, there exists $d\in D$ such that $d\in \ua G$, which contradicts $D\subseteq\bigcup_{F\in\mathcal{F}}(F\backslash\ua G)$.
\end{proof}

Similar to Proposition \ref{equivalent condition of exact} for exact posets, we have the following result.

\begin{prop}\label{equivalent condition 1 of quasiexact}
A dcpo $P$ is quasiexact if and only if for any $x\in P$, there exists a directed subset $\mathcal{F}\subseteq \fin_{w}(x)$ such that $\bigcap_{F\in \mathcal{F}}\ua F=\ua x$.
\end{prop}

\begin{proof}
It is enough to prove the sufficiency. Let $x\in P$ and $\mathcal{F}\subseteq\fin_{w}(x)$ be a directed family with $\bigcap_{F\in\mathcal{F}}\ua F=\ua x$. Since  $\ua x=\bigcap_{F\in\mathcal{F}}\ua F\supseteq \bigcap_{G\in\fin_{w}(x)}\ua G \supseteq \ua x$, it follows that $\bigcap_{G\in\fin_{w}(x)}\ua G= \ua x$. It remains to show that $\fin_{w}(x)$ is directed.
Let $G_{1}, G_{2}\in\fin_{w}(x)$. By Lemma \ref{equivalence of weak way-below}, we have $\da x\cap G_{1}\ll_{w}x$ and $\da x\cap G_{2}\ll_{w}x$. By Lemma \ref{F intersection down x is directed}, $\{\da x\cap F: F\in \mathcal{F}\}$ is a directed collection of nonempty finite subsets of $\da x$. Furthermore, $\bigcap\{\ua(\da x\cap F): F\in \mathcal{F}\}\subseteq\bigcap\{\ua F: F\in\mathcal{F}\}=\ua x$. By Lemma \ref{Lemma according to Rudin's Lemma}, there exists $F_{1}, F_{2}\in \mathcal{F}$ such that $\da x\cap F_{1}\subseteq\ua(\da x\cap G_{1})$ and $\da x\cap F_{2}\subseteq\ua(\da x\cap G_{2})$. Choose $F_{3}\in \mathcal{F}$ such that $\da x\cap F_{3}\subseteq\ua(\da x\cap F_{1})\cap \ua(\da x\cap F_{2})$. Note that $\ua(\da x\cap F_{1})\cap \ua(\da x\cap F_{2})\subseteq\ua(\da x\cap G_{1})\cap\ua(\da x\cap G_{2})\subseteq \ua G_{1}\cap\ua G_{2}$. It follows that $\da x\cap F_{3}\subseteq \ua G_{1}\cap\ua G_{2}$. Applying Lemma \ref{equivalence of weak way-below} again, we can conclude $\da x\cap F_{3}\in\fin_{w}(x)$, whence $\fin_{w}(x)$ is directed. So $P$ is quasiexact.
\end{proof}

We write $\fin_{(bl)w}(x)=\{F: F\in \fin_{w}(x), F\subseteq\da x\}$. For any $F\in\fin_{w}(x)$, we have $\da x\cap F\ll_{w}x$ by Lemma \ref{equivalence of weak way-below}. It follows that $\da x\cap F\in\fin_{(bl)w}(x)$. Note that $F=\da x\cap F$ for any $F\in\fin_{(bl)w}(x)$, so $\fin_{(bl)w}(x)=\{\da x\cap F: F\in\fin_{(bl)w}(x)\}\subseteq\{\da x\cap F: F\in\fin_{w}(x)\}$. Therefore, $\fin_{(bl)w}(x)=\{\da x\cap F: F\in\fin_{w}(x)\}$.

Now, we give another characterization of quasiexact dcpos.

\begin{prop}\label{equivalent condition 2 of quasiexact}
A dcpo $P$ is quasiexact if and only if for any $x\in P$, the collection $\fin_{(bl)w}(x)$ is directed and $\bigcap_{F\in \fin_{(bl)w}(x)}\ua F=\ua x$.
\end{prop}

\begin{proof}
Note that $\fin_{(bl)w}(x)\subseteq\fin_{w}(x)$. The sufficiency is immediately from Proposition \ref{equivalent condition 1 of quasiexact}. For the necessity, assume that $P$ is quasiexact. Then for each $x\in P$,  $\fin_{w}(x)$ is directed and $\bigcap\{\ua G: G\in\fin_{w}(x)\}=\ua x$. By Lemma \ref{F intersection down x is directed}, we have	$\fin_{(bl)w}(x)$ is also directed. Note also that $\fin_{(bl)w}(x)=\{\da x\cap G: G\in\fin_{w}(x)\}$. It follows that $\bigcap_{F\in \fin_{(bl)w}(x)}\ua F=\bigcap\{\ua(\da x\cap G): G\in\fin_{w}(x)\}\subseteq\bigcap\{\ua G: G\in\fin_{w}(x)\}=\da x$. On the other hand, $\fin_{(bl)w}(x)\subseteq\fin_{w}(x)$, so $\bigcap_{F\in \fin_{(bl)w}(x)}\ua F\supseteq \bigcap\{\ua G: G\in\fin_{w}(x)\}=\da x$.
Hence, $\bigcap\{\ua F: F\in\fin_{(bl)w}(x)\}=\ua x$.
\end{proof}

Applying Propositions \ref{equivalent condition 1 of quasiexact} and \ref{equivalent condition 2 of quasiexact}, we can derive the following result.

\begin{thm}\label{5 equivalent statements of quasiexact}
For any dcpo $P$, the following statements are equivalent:
\begin{enumerate}
\renewcommand\labelenumi{($\alph{enumi}$)}
  \item $P$ is quasiexact;
  \item for any $x\in P$, the collection $\fin_{(bl)w}(x)$ is directed and $\bigcap\{\ua F: F\in\fin_{(bl)w}(x)\}=\ua x$;
  \item for any $x\in P$, there exists a directed subset $\mathcal{F}\subseteq\fin_{w}(x)$ such that $\bigcap_{F\in \mathcal{F}}\ua F=\ua x$;
  \item for any $x\in P$, there exists a directed subset $\mathcal{F}\subseteq\fin_{(bl)w}(x)$ such that $\bigcap_{F\in \mathcal{F}}\ua F=\ua x$.
\end{enumerate}
\end{thm}

\begin{prop}\label{product of quasiexact posets}
The cartesian product $\prod_{i\in I}P_{i}$ of a family of quasiexact dcpos is a quasiexact dcpo, provided that at most finitely many do not have a bottom element $\perp_{i}$.
\end{prop}
\begin{proof}
Obviously, $\prod_{i\in I}P_{i}$ is a dcpo. Let $I_{0}$ be a finite subset of $I$ such that $P_{i}$ has $\perp_{i}$ for any $i\in I\backslash I_{0}$ and $\Gamma=\{J: I_{0}\subseteq J\subseteq I, J~\mbox{is finite}\}$. Consider each $x=(x_{i})_{i\in I}\in\prod_{i\in I}P_{i}$. Take $\mathcal{F}_{J}=\{\prod_{j\in J}F_{j}\times\prod_{i\in I\backslash J}\{\perp_{i}\}: F_{j}\in\fin_{w}(x_{j}), j\in J\}$ for any $J\in\Gamma$ and $\mathcal{F}=\bigcup_{J\in\Gamma}\mathcal{F}_{J}$. Then $\mathcal{F}\subseteq \fin_{w}(x)$.
\begin{description}
  \item[Claim 1.] $\bigcap\{\ua F: F\in\mathcal{F}\}=\ua x$.
  
  If $y=(y_{i})_{i\in I}\notin\ua x$, then there exists $j_{0}\in I$ such that $y_{j_{0}}\notin \ua x_{j_{0}}$. By the quasiexactness of $P_{j_{0}}$, there exists $F_{j_{0}}\in\fin_{w}(x_{j_{0}})$ such that $y_{j_{0}}\notin \ua F_{j_{0}}$. Let $J_{0}=I_{0}\cup\{j_{0}\}$. Let $F_{0}=F_{j_{0}}\times\prod_{j\in J_{0}\backslash\{j_{0}\}}F_{j}\times\prod_{i\in I\backslash J}\{\perp_{i}\}$. Then $F_{0}\in \mathcal{F}_{J_{0}}\subseteq\mathcal{F}$, but $y\notin \ua F_{0}$. It follows that $y\notin \ua F_{j_{0}}\times\prod_{j\in J_{0}\backslash\{j_{0}\}}\ua F_{j}\times\prod_{i\in I\backslash J}P_{i}=\ua F_{0}$. Therefore, $\bigcap\{\ua F: F\in\mathcal{F}\}\subseteq\ua x$.
 It follows that $y\notin\bigcap\{\ua F: F\in\mathcal{F}\}$, whence $\bigcap\{\ua F: F\in\mathcal{F}\}=\ua x$.
  
  \item[Claim 2.] $\mathcal{F}$ is a directed subset of $\fin_{w}(x)$.

  Let $F_{1}, F_{2}\in \mathcal{F}$. Suppose that $F_{1}=\prod_{j'\in J_{1}}F_{j'_{1}}\times\prod_{i\in I\backslash J}\{\perp_{i}\}\in\mathcal{F}_{J_{1}}$ and $F_{2}=\prod_{j''\in J_{2}}F_{j''_{2}}\times\prod_{i\in I\backslash J}\{\perp_{i}\}\in\mathcal{F}_{J_{2}}$, where $J_{1}, J_{2}$ are finite and $I_{0}\subseteq J_{1}, J_{2}\subseteq I$, moreover, $F_{j'_{1}}\in\fin_{w}(x_{j'})$, $F_{j''_{2}}\in\fin_{w}(x_{j''})$ for any $j'\in J_{1}$, $j''\in J_{2}$. For any $j\in J_{1}\cap J_{2}$, $F_{j_{1}}, F_{j_{2}}\in\fin_{w}(x_{j})$. Note that $\fin_{w}(x_{j})$ is directed. Choose $F_{j_{3}}\in\fin_{w}(x_{j})$ such that $F_{j_{3}}\subseteq\ua F_{j_{1}}\cap\ua F_{j_{2}}$. For any $k'\in J_{1}\backslash J_{2}$ and $k''\in J_{2}\backslash J_{1}$, obviously, $F_{k'_{1}}\subseteq \ua F_{k'_{1}}\cap \ua \bot_{k'}$ and $F_{k''_{2}}\subseteq \ua F_{k''_{2}}\cap \ua \bot_{k''}$. Put $J_{3}=J_{1}\cup J_{2}$ and
 \vskip0.2cm
 \begin{center}
 $F_{3}=\prod_{j\in J_{1}\cap J_{2}}F_{j_{3}}\times
  \prod_{k'\in J_{1}\backslash J_{2}}F_{k'}\times
  \prod_{k''\in J_{2}\backslash J_{1}}F_{k''}\times
  \prod_{i\in I\backslash J_{3}}\{\perp_{i}\}.$
 \end{center}
 \vskip0.2cm
 Then $F_{3}\in \mathcal{F}_{J_{3}}$, $J_{3}$ is finite and $I_{0}\subseteq J_{3}\subseteq I$. Obviously, $F_{3}\subseteq \ua F_{1}\cap \ua F_{2}$.
\end{description}
By Proposition \ref{equivalent condition 1 of quasiexact}, $\prod_{J}P_{j\in J}$ is quasiexact.
\end{proof}

Some relationships among quasiexact dcpos, weak domains and quasicontinuous domains are  shown in  the following result.

\begin{prop}\label{Proposition-5.1}
\begin{enumerate}
\renewcommand{\labelenumi}{(\theenumi)}
  \item Every exact dcpo is quasiexact. Hence, every weak domain is a quasiexact dcpo.
  \item Every quasicontinuous domain is a quasiexact dcpo.
\end{enumerate}
\end{prop}

\begin{proof}
(i) Let $P$ be an exact poset. For any $x\in P$, let $\mathcal{F}=\{\{d\}: d\in\dda_{w} x\}$. Note that $\dda_{w}x$ is directed with $\bigvee\dda_{w}x=x$. Then $\mathcal{F}$ is directed and
$$\begin{array}{ll}
& \bigcap\{\ua F: F\in\mathcal{F}\}\\
= & \bigcap\{\ua d: d\in\dda_{w} x\}\\
= & \ua(\bigvee\dda_{w} x)\\
= & \ua x.
\end{array}$$
By Proposition \ref{equivalent condition 1 of quasiexact}, $P$ is quasiexact.

(ii) Assume that $P$ is a quasicontinuous domain. Obviously, $P$ is a dcpo. Let $x\in P$. By the assumption and Lemma \ref{Lemma-Mushburn} (1), we get that $\fin(x)$ is a directed subset of $\fin_{w}(x)$ with $\bigcap\{\ua F: F\in\fin(x)\}=\ua x$. By Proposition \ref{equivalent condition 1 of quasiexact}, $P$ is quasiexact.
\end{proof}

Quasiexact posets have some weak and quasicontinuous domain-like features, so this terminology seems appropriate.

Shen et al. \cite{Shen-Wu-Zhao-2019} used two examples to show that quasicontinuous domains need not be weak domains, and that weak domains need not be quasicontinuous domains. Precisely, Example 3.11 in \cite{Shen-Wu-Zhao-2019} shows that quasicontinuous domains are not necessarily exact, and Example 3.12 in \cite{Shen-Wu-Zhao-2019} shows that weak domains are not necessarily quasicontinuous.

Here, we shall elaborate some more detailed relationships than those in \cite{Shen-Wu-Zhao-2019}.

The following example shows that quasicontinuous domains are not necessarily weakly increasing. Since every quasicontinuous domain is quasiexact, this example also shows that quasiexact posets need not to be  exact or  weakly increasing.

\begin{example}\label{a quasicontinuous domain, exact, not weakly increasing}
\emph{Let $P=\{a, b, c, d\}\cup\{x_{n}: n\in \mathbb{N}\}$ with the order
\begin{enumerate}
\renewcommand{\labelenumi}{(\roman{enumi})}
\item $a< b< c< d$;
\item $x_{m}<x_{n}$ whenever $m< n$ and $m,n\in\mathbb{N}$;
\item $x_{n}< c$ for any $n\in\mathbb{N}$,
\end{enumerate}
where $\{a, b, c, d\}\cap\{x_{n}: n\in \mathbb{N}\}=\varnothing$, see Figure 1.}
\begin{figure}[htbp!]\label{directed-4-points}
  \centering
  \includegraphics[width=7cm]{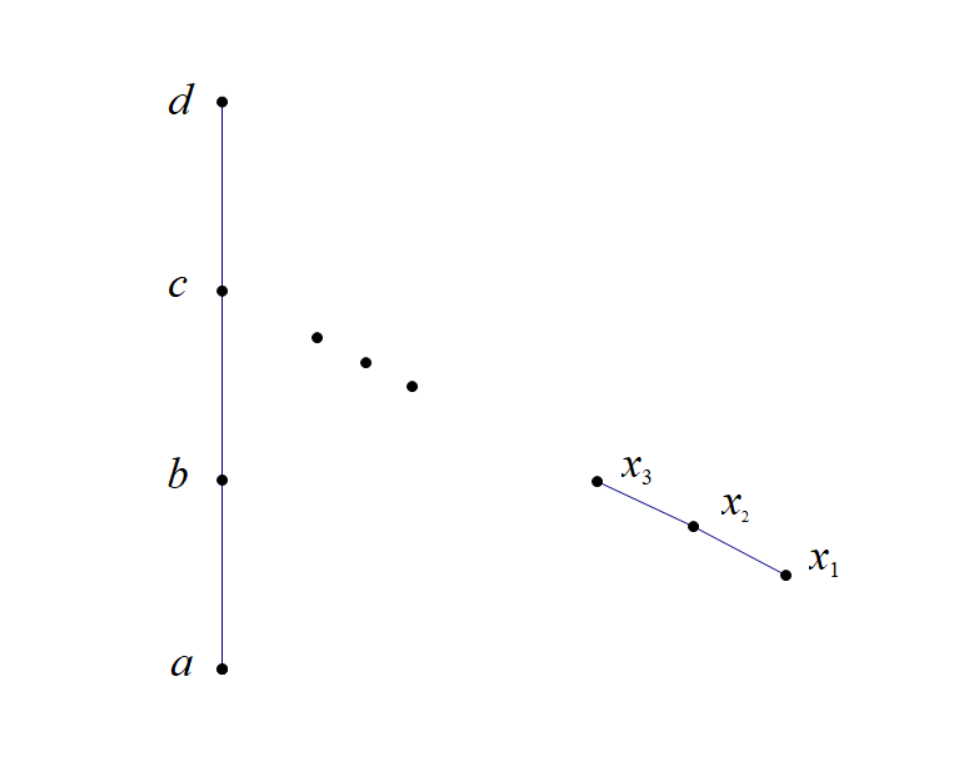}\\
  \caption{An exact quasicontinuous domain in which $\ll_{w}$ is not weakly increasing}
\end{figure}

\emph{In this example, we can trivially check the following facts.
\begin{enumerate}
\renewcommand{\labelenumi}{1-(\roman{enumi})}
  \item[(\rmnum{4})] If $x\in\{a, b, c\}$, then $F\in\fin(x)$ if and only if $F\cap\{x_{n}: n\in\mathbb{N}\}\neq\varnothing;$
  \item[(\rmnum{5})] If $x\in \{x_{n}: n\in\mathbb{N}\}\cup\{d\}$, then $F\in\fin(x)$ if and only if $F\cap\da x\neq\varnothing.$
\end{enumerate}
It is also trivial to check that $\fin(x)$ is directed with $\bigcap\{\ua F: F\in\fin(x)\}=\ua x$ for every $x\in P$, whence $P$ is a quasicontinuous domain. However, the relation $\ll_{w}$ is not weakly increasing. Note that $a\ll_{w}b\leq c\ll_{w}d$. Consider the directed set $D=\{x_{n}: n\in \mathbb{N}\}$. Obviously, $\bigvee D=c$, but $x_{n}\ngeqslant a$ for any $n\in\mathbb{N}$. Hence, $a\ll_{w}c$ does not hold. For the point $c\in P$, it is trivial to verify that $\dda c=\{x_{n}: n\in\mathbb{N}\}$, whence $\bigvee \dda c=c$. Hence, $P$ is exact.}
\end{example}

The following example shows that quasicontinuous domains need not be exact, even when the relation $\ll_{w}$ is weakly increasing.

\begin{ex}\label{quasiexact dcpo, not weak domain}
\emph{Let $P=(\mathbb{N}\times\{1, 2\})\cup\{\top\}$ and define an order by the following rules:
\begin{enumerate}
\renewcommand{\labelenumi}{(\roman{enumi})}
  \item $(m, i)<(n, i)$ if $m<n$ for all $m, n\in\mathbb{N}$ and $i=1, 2$;
  \item $(n, i)<\top$ for all $n\in \mathbb{N}$ and $i=1, 2$ (see Figure 2).
\end{enumerate}
}
\begin{figure}[htbp!]\label{top-element}
  \centering
  \includegraphics[width=8.5cm]{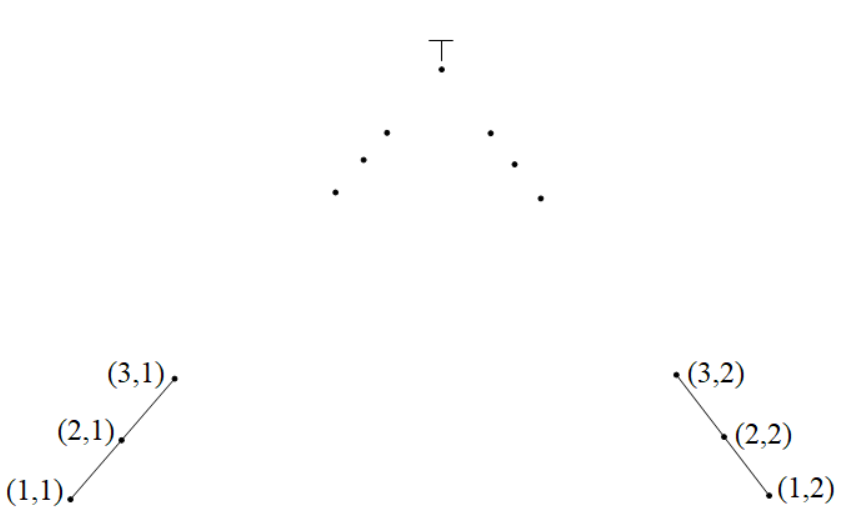}\\
  \caption{A non-exact quasicontinuous domain with $\ll_{w}$ weakly increasing}
\end{figure}
\emph{In this example, for any finite set $F\subseteq P$, $F\ll (m, 1)$ if and only if $F\cap\da (m, 1)\neq \varnothing$ and $F\cap\{(n, 2): n\in\mathbb{N}\}\neq \varnothing$;  $F\ll (m, 2)$ if and only if $F\cap\da (m, 2)\neq \varnothing$ and $F\cap\{(n, 1): n\in\mathbb{N}\}\neq \varnothing$; $F\ll\top$ if and only if $F\cap\{(n, 1): n\in \mathbb{N}\}\neq\varnothing$ and $F\cap\{(n, 2): n\in \mathbb{N}\}\neq\varnothing$. It is easy to verify that $P$ is a quasicontinuous domain. Thus, $P$ is quasiexact. Note also that for any finite set $F$, $F\ll_{w} (n, i)$ if and only if $F\cap\da(n, i)\neq\varnothing$, where $i=1, 2$; $F\ll_{w}\top$ if and only if $F\cap\{(n, 1): n\in \mathbb{N}\}\neq\varnothing$ and $F\cap\{(n, 2): n\in \mathbb{N}\}\neq\varnothing$. It is trivial to verify that the relation $\ll_{w}$ is weakly increasing. However, $P$ is not exact. Consider the top element $\top$. For any $(m, 1)$ and $m\in \mathbb{N}$, consider the directed set $D=\{(n, 2): n\in \mathbb{N}\}$. Then $\bigvee D=\top\geq(m, 1)$, but $(n, 2)\ngeqslant(m, 1)$ for any $(n, 2)$ and $n\in\mathbb{N}$. Thus, $(m, 1)\notin\dda_{w}\top$. Similarly, we can verify that $(m, 2)\notin\dda_{w}\top$ and $\top\notin\dda_{w}\top$ for any $i\in \mathbb{N}$. It follows that $\dda_{w}\top=\varnothing$. Hence, $P$ is not exact.}
\end{ex}

Examples \ref{a quasicontinuous domain, exact, not weakly increasing} and \ref{quasiexact dcpo, not weak domain} in the current paper also shows that quasiexacnesst does not imply stronger properties of $\ll_{w}$.

We take Johnstone's dcpo  to illustrate that weak domains are not necessarily quasicontinuous domain, hence quasiexact dcpos with the relation $\ll_{w}$ weakly increasing are not necessarily quasicontinuous.

\begin{ex}\label{Johnstone space}
\emph{({\bf Johnstone space}) Let $\mathbb{J}=\mathbb{N}\times(\mathbb{N}\cup\{\omega\})$ with ordering defined by}
\begin{enumerate}
\renewcommand{\labelenumi}{(\roman{enumi})}
  \item $(a, m)<(a, n)$ if $m<n$ for all $a, m, n\in\mathbb{N}$;
  \item $(a, m)<(b, \omega)$ if $m\leq b$ for all $a, b, m\in\mathbb{N}$.
\end{enumerate}
\emph{In this example, one can easily check the following fact: For any $m, n\in\mathbb{N}$, $\dda_{w}(m, n)=\da(m, n)$ and $\dda_{w}(m, \omega)=\{(m, n): n\in \mathbb{N}\}$, whence, $\bigvee\dda_{w}(m, n)=(m, n)$ and $\bigvee\dda_{w}(m, \omega)=(m, \omega)$. Thus, $\mathbb{J}$ is exact. It is trivial to show that $\ll_{w}$ is weakly increasing. However, $\mathbb{J}$ is not quasicontinuous. By the result in \cite[Exercise 8.2.14]{Goubault-book-2013}, $\mathbb{J}$ is a dcpo with the Scott space $\Sigma(\mathbb{J})$ non-sober. It follows that $P$ is not quasicontinuous.}
\end{ex}

\section{The wf topology and moderate meet continuity}\label{}
In this section, we investigate the links between the quasiexactness and some topologies on posets.


In \cite{Mushburn-2007}, when the family $\{\dua_{w} x: x\in P\}$ generates a topology on a poset $P$, then this topology is called the \emph{weak way-below topology} (\emph{wwb topology}, for short), denoted by $\tau_{wwb}(P)$. The topological space $(P, \tau_{wwb}(P))$  is simply written as $(P, \tau_{wwb})$.

For the following result, we refer the reader to the Theorems 3.8 and 3.10 in \cite{Mushburn-2007}.

\begin{lem}\label{wwb topology and Scott topology}\emph{\cite{Mushburn-2007}}\label{Mushburn's result}
Let $P$ be an exact poset. Then $\{\dua_{w} x: x\in P\}$ generates the wwb topology, which is finer than the Scott topology.
\end{lem}

For any subset $F\subseteq P$, we write $\dua_{w}F=\{a\in P: F\ll_{w}a\}$. Whenever $F=\{x\}$ for some $x\in P$, we replace $\dua_{w}\{x\}$ with $\dua_{w}x$.

\begin{lem}\label{basis of wf topology}
If $P$ is a quasiexact poset, then $\{\dua_{w}F: F\subseteq P, F~\emph{\mbox{is finite}}\}$ generates a topology on $P$.
\end{lem}

\begin{proof}
For any $a\in P$, $\fin_{w}(a)$ is directed with $\bigcap\{\ua F: F\in\fin_{w}(a)\}=\ua a$. Thus, $\fin_{w}(a)\neq \varnothing$. Choose arbitrarily $F\in\fin_{w}(a)$, then $a\in \dua_{w}F$.

Let $F_{1}, F_{2}$ be finite sets in $P$ with $b\in \dua_{w}F_{1}\cap\dua_{w}F_{2}$, i.e., $F_{1}\ll_{w} b$ and $F_{2}\ll_{w} b$. Note that $\fin_{w}(b)$ is directed. There exists $F_{3}\in\fin_{w}(b)$ such that $F_{3}\subseteq \ua F_{1}\cap\ua F_{2}$. For any $e\in P$, if $F_{3}\ll_{w}e$, then $F_{1}, F_{2}\ll_{w}e$, so $b\in\dua_{w}F_{3}\subseteq\dua_{w}F_{1}\cap\dua_{w}F_{2}$.
\end{proof}

Whenever the family $\{\dua_{w}F: F\subseteq P, F~\mbox{is finite}\}$ generates a topology on $P$, we call it the \emph{weak way-below finitely determined topology} (briefly, \emph{wf topology}) on $P$, denoted by $\tau_{wf}(P)$. The topological space $(P, \tau_{wf}(P))$  will be simply  written as $(P, \tau_{wf})$.

\begin{rem}\label{Remark-4.1}
\begin{enumerate}
\renewcommand{\labelenumi}{(\theenumi)}
  \item  For any quasiexact poset $P$, we have $\dua_{w}F\subseteq\Int_{\tau_{wf}}(\ua F)$ for every $F\in\mathcal{P}^{*}(P)$.
  \item If a poset $P$ admits both the wwb topology and the wf topology, then $\tau_{wwb}(P)\subseteq\tau_{wf}(P)$.
\end{enumerate}
\end{rem}

By Lemma \ref{wwb topology and Scott topology} and Remark \ref{Remark-4.1} (2), we have the following result.

\begin{lem}\label{wf is finer that Scott topology}
If $P$ is an exact poset, then $\sigma(P)\subseteq\tau_{wf}(P)$.
\end{lem}

Note that every quasicontinuous dcpo admits the wf topology. However, the following example shows that quasicontinuous dcpos do not necessarily admit the wwb topology.

\begin{ex}\label{Example-4.3}
\emph{Consider the set $P=(\mathbb{N}\times\{1, 2\})\cup\{\top\}$ with the order defined in Example \ref{quasiexact dcpo, not weak domain}, which shows that $P$ is quasicontinuous dcpo. It is trivial to check the following facts:}
\emph{\begin{enumerate}
\renewcommand{\labelenumi}{(\roman{enumi})}
  \item $\dua_{w}(n, i)=\{(m, i): m\geq n, m\in\mathbb{N}\}$ for any $n\in\mathbb{N}$ and $i=1, 2$;
  \item $\dua_{w}\top=\varnothing$.
\end{enumerate}}
\noindent\emph{Note that the family $\{\dua_{w}x: x\in P\}$ can not cover $\top$. Thus, $\{\dua_{w}x: x\in P\}$ can not generate the wwb topology.}
\end{ex}

Mushburn \cite{Mushburn-2007} constructed an example to show the wwb topology can be strictly finer than the Scott topology.

\begin{prop}\label{Proposition-5.6}
A poset $P$ is quasiexact if  for any nonempty $H\subseteq P$ and any $x\in P$, $H\ll_{w}x$ implies there exists a finite $F\subseteq \ua H$ such that $F\ll_{w}x$.
\end{prop}

\begin{proof}
It is trivial to verify that $\fin_{w}(\bot)$ is directed with $\bigcap\{\ua F: F\ll_{w}\bot\}=\bot$ whenever the bottom element $\bot$ exists. Without loss of generality, assume that $x\neq \bot$. Note that $\ua x\subsetneqq P$. For any $y\notin \ua x$, we show $P\sm\da y\ll_{w}x$. Otherwise, there exists a directed set $D$ with $\bigvee D=x$, but $D\cap (P\sm\da y)=\varnothing$ since $P\sm\da y$ is an upper set. Thus, $D\subseteq\da y$. It follows that $x=\bigvee D\leq y$, a contradiction. By the hypothesis, there exists a finite set $F\subseteq P\sm\da y$ such that $F\ll_{w}x$. This implies that $\fin_{w}(x)\neq\varnothing$. Furthermore, note that $\ua F\subseteq P\sm\da y$. Then $\ua F\cap\da y=\varnothing$. It follows that $y\notin\ua F$. Hence, $\bigcap\{\ua F: F\in\fin_{w}(x)\}\subseteq \ua x$. It remains to show that $\fin_{w}(x)$ is directed. For this, let $F_{i}\in\fin_{w}(x)$ and $i=1, 2$. We show that $\ua F_{1}\cap \ua F_{2}\ll_{w}x$. For any directed set $D$, if $\bigvee D=x$, then $D\cap\ua F_{i}\neq\varnothing$, i.e, there exists $d_{i}\in D$ and $e_{i}\in F_{i}$ such that $d_{i}\geq e_{i}$ for $i=1, 2$. Choose $d_{3}\in D$ such that $d_{3}\geq d_{i}$ for $i=1, 2$. Then $d_{3}\geq e_{i}$ for $i=1, 2$. It follows that $\ua F_{1}\cap \ua F_{2}\ll_{w}x$. Note that $\ua F_{1}\cap \ua F_{2}$ is an upper set. By the hypothesis, there exists a finite set $F_{3}\subseteq \ua F_{1}\cap \ua F_{2}$ such that $F_{3}\ll_{w}x$. Therefore, $\fin_{w}(x)$ is directed.
\end{proof}

\begin{prop}
Any poset that has no infinite antichain is quasiexact.
\end{prop}

\begin{proof}
Assume that $P$ is a poset with no infinite antichain. For any nonempty $H\subseteq P$ and any $x\in P$, if  $H\ll_{w}x$, then by Zorn's lemma we can pick a maximal antichain $A$ in $\ua H\sm \ua x$. Note that $F=A\cup\{x\}$ is finite, contained in $\ua H$. We show $F\ll_{w}x$. For this, let $D$ be a directed set with $\bigvee D=x$. Since $H\ll_{w}x$, eventually, $D$ is in $\ua H$. If $d=x$ for some $d\in D$, then $d\in D\cap \ua F$. Otherwise, $d\notin \ua x$ for any $d\in D$. If there exists $d\in D\cap\ua H$ such that $d\notin\ua A\cup\da A$, then $d$ is incomparable with each $a\in A$. Note that $d\notin \ua x$ and that $A$ is an antichain in $\ua H\sm \ua x$. It follows that $A\cup\{d\}$ is also an antichain in $\ua H\sm \ua x$, contradicting the maximality of $A$. So $D\cap\ua H\subseteq\ua A\cup\da A$. It cannot be the case $D\subseteq\da A$ since $\bigvee D\geq x$. Note that every $d$ large enough in $D$ is in $\ua H$, hence in $\ua A\cup\da A$, but not all of them in $\da A$, and those that are not must be in $\ua A$. Thus, $D\cap\ua F\supseteq D\cap\ua A\neq\varnothing$. Applying Proposition \ref{Proposition-5.6}, we conclude that $P$ is quasiexact.
\end{proof}

Following the characterization of meet continuous posets by means of Scott topology, we define the following meet continuity.

\begin{defi}
A poset $P$ admitting the wf topology is said to be \emph{moderately meet continuous} if for any $x\in P$ and directed subset $D$, $x\leq \bigvee D$ implies that $x\in \cl_{\tau_{wf}}(\da D\cap\da x)$.
\end{defi}

\begin{lem}\label{wf topology}
Let $P$ be a moderately meet continuous poset. Then $\Int_{\tau_{wf}}(\ua F)\subseteq \bigcup\{\dua_{w}x: x\in F\}$ for any $F\in\mathcal{P}^{*}(P)$.
\end{lem}

\begin{proof}
For convenience, we write $\Int_{\tau_{wf}}(\ua F)=U$ and let $F=\{x_{1}, x_{2}, \cdots, x_{n}\}$. Assume $x\in U$, but $y\notin\bigcup\{\dua_{w}x: x\in F\}$, that is, $x_{i}\ll_{w}y$ does not hold for any $i=1, 2, \cdots, n$. Thus, there exists a directed subset $D_{i}\subseteq P$ such that $\bigvee D_{i}=y$, but $x_{i}\notin\da D_{i}$. By the moderately meet continuity, $y\in\cl_{\tau_{wf}}(\da D_{1}\cap\da y)$. Since $y\in U$ and $U\in\tau_{wf}(P)$, then $U\cap (\da D_{1}\cap\da y)\neq\varnothing$. Pick $z_{1}\in U\cap \da D_{1}\cap\da y$. Note that $z_{1}\leq y=\bigvee D_{2}$. By the moderately meet continuity again, we have $z_{1}\in\cl_{\tau_{wf}}(\da D_{2}\cap\da z_{1})$, whence $U\cap\da D_{2}\cap\da z_{1}\neq\varnothing$. Pick $z_{2}\in U\cap\da D_{2}\cap\da z_{1}$. The rest can be done in the same manner. In other words, we can pick $z_{i}\in U\cap\da D_{i}\cap\da z_{i-1}$ for any $i$ ($1\leq i\leq n$). Note that $z_{i}\leq z_{i-1}$ and $z_{i-1}\in \da D_{i-1}$ for all $i$. It follows that $z_{n}\in \da D_{i}$ for all $i$. Thus, $z_{n}\in\bigcap^{n}_{i=1}\da D_{i}$. Also note that $z_{n}\in U=\mbox{int}_{\tau_{wf}}(\ua F)\subseteq\ua F$. Hence, there exists $i_{0}$ ($1\leq i_{0}\leq n$) such that $x_{i_{0}}\leq z_{n}\in \da D_{i_{0}}$, showing that $x_{i_{0}}\in\da D_{i_{0}}$, a contradiction. Therefore, $\Int_{\tau_{wf}}(\ua F)\subseteq\bigcup\{\dua_{w}x: x\in F\}$.
\end{proof}

Applying Lemmas \ref{Lemma-3.2} (3), \ref{wf topology} and Remark \ref{Remark-4.1} , we derive the following conclusion.

\begin{cor}\label{Corollary-4.9}
Let $P$ be a moderately meet continuous quasiexact poset. Then
\vskip0.3cm
\centerline{$\dua_{w}F=\bigcup\{\dua_{w}x: x\in F\}=\Int_{\tau_{wf}}(\ua F)$}
\vskip0.3cm
\noindent for any $F\in\mathcal{P}^{*}(P)$.
\end{cor}

By Remark \ref{Remark-4.1} (2) and Corollary \ref{Corollary-4.9}, we deduce the following result strengthening Mushburn's result in \cite[Theorem 10]{Mushburn-2007}.

\begin{thm}
If $P$ is a moderately meet continuous quasiexact poset, then $\sigma(P)\subseteq\tau_{wwb}(P)=\tau_{wf}(P)$.
\end{thm}

\begin{thm}\label{moderately meet continuous quasiexact posets are meet continuous and exact}
Every moderately meet continuous quasiexact poset is exact and meet continuous.
\end{thm}

\begin{proof}
Let $P$ be a moderately meet continuous quasiexact. To show that $P$ is exact, we consider any $a\in P$.
\begin{description}
  \item[Claim 1.] $\dda_{w}a$ is directed.

  Let $y_{1}, y_{2}\in \dda_{w}a$. Then $\{y_{1}\}, \{y_{2}\}\in\fin_{bl(w)}(a)$. By the hypothesis, there exists $F\in\fin_{bl(w)}(a)$ such that $F\subseteq \ua y_{1}\cap \ua y_{2}$. By Corollary \ref{Corollary-4.9}, we have $a\in\bigcup\{\dua_{w}x: x\in F\}$. So there exists $x\in F$ such that $x\ll_{w}a$. Note that $F\subseteq \ua y_{1}\cap\ua y_{2}$, then $x\geq y_{1}$ and $x\geq y_{2}$.

  \item[Claim 2.] $\bigcap\{\ua y: y\in\dda_{w}a\}=\ua a$.

Obviously, $\bigcap\{\ua y: y\in\dda_{w}a\}\supseteq\ua a$. For any $F\in\fin_{bl(w)}(a)$, by Corollary \ref{Corollary-4.9}, we can pick $a_{F}\in F$ such that $a_{F}\ll_{w}a$. Then $\{\ua y_{F}: F\in\fin_{bl(w)}(x)\}\subseteq\{\ua y: y\in\dda_{w}x\}$. It follows that $\bigcap\{\ua y: y\in\dda_{w}x\}\subseteq\bigcap\{\ua y_{F}: F\in\fin_{bl(w)}(x)\}\subseteq\bigcap\{\ua F: F\in\fin^{w}_{bl}(x)\}=\ua x$. Hence, $\bigcap\{\ua y: y\in\dda_{w}a\}=\ua a$. Note that $\ua(\bigvee \dda_{w}a)=\bigcap\{\ua y: y\in\dda_{w}a\}$. So $\bigvee \dda_{w}a=a$.
\end{description}
By Proposition \ref{equivalent condition of exact}, we conclude that $P$ is exact.

For any $x\in P$ and directed set $D\subseteq P$, if $x\in\cl_{\tau_{wf}}(\da D\cap\da x)$, then by Lemma \ref{wf is finer that Scott topology} and the definition of closure, we have $\cl_{\tau_{wf}}(\da D\cap\da x)\subseteq\cl_{\sigma}(\da D\cap\da x)$. Therefore, $P$ is meet continuous.
\end{proof}

Shen et al. \cite{Shen-Wu-Zhao-2019} showed that every meet continuous weak domain is a domain. By Proposition \ref{moderately meet continuous quasiexact posets are meet continuous and exact}, we conclude the following result.

\begin{cor}
A poset $P$ is a domain if $P$ is a moderately meet continuous quasiexact dcpo with the relation $\ll_{w}$ weakly increasing.
\end{cor}

In this paper, we explored the quasiexact posets, parallel to the quasicontinuous posets. A new topology, the wf topology  on posets is introduced and used to define the moderately meet continuous posets. Although several results on such structures have been obtained, we still have basic problems to solve as illustrated below.

\noindent{\bf Problem 1.} What is the property $p$ such that a poset $P$ is exact if and only if it is quasiexact and has property $p$?

It is known that a poset $P$ is continuous if and only if it is quasicontinuous and meet continuous \cite{Mao-Xu-2006}. However, we still do not have a similar result for exact posets. It is only proved that every exact poset is quasiexact. 

\noindent{\bf Problem 2.} Under what conditions, a quasiexact dcpo is quasicontinuous?

At the moment we just know that every quasicontinuous dcpo is quasiexact. It would be ideal if we could find a property $q$ such that a dcpo is quasicontinuous if and only if it is quasiexact and has property $q$.

\section*{Acknowledgements}
We thank Professor Guohua Wu in Nanyang Technological University for his valuable suggestions.
\vskip 2mm

\bibliographystyle{./entics}

\end{document}